\documentclass[11pt]{amsart}
\usepackage{amsthm}

\usepackage{amsmath}
\usepackage{amssymb}
\usepackage{amsfonts}
\usepackage{color}
\usepackage{xspace}

\newcommand{\ignore}[1]{}

\newtheorem{theorem}{Theorem}

\newtheorem{lemma}[theorem]{Lemma}

\title{Elementary Proofs of Some Stirling Bounds}
\author{N. H. Bshouty}
\address{{\bf Nader H. Bshouty} Dept. of Computer Science. Technion,  Haifa, 32000}
\author{V. E. Bshouty-Hurani}
\address{{\bf Vivian E. Bshouty-Hurani}. The Arab Orthodox College.  Haifa.}
\author{G. Haddad }
\address{{\bf George Haddad}. The Arab Orthodox College. Grade 10. Haifa.}
\author{T. Hashem}
\address{{\bf Thomas Hashem}. Sister of St. Joseph High School. Grade 11. Nazareth.}
\author{F. Khoury}
\address{{\bf Fadi Khoury}. Sister of Nazareth High School. Grade 10. P.O.B. 9422, Haifa, 35661.}
\author{O. Sharafy}
\address{{\bf Omar Sharafy}. The Arab Orthodox College. Grade 10. Haifa.}

\usepackage{lipsum}
\begin{document}

\begin{abstract}
We give elementary proofs of several Stirling's precise bounds. We first improve all the precise bounds
from the literature and give new precise bounds.
In particular, we show that for all $n\ge 8$
$$\sqrt{2\pi n}\left(\frac{n}{e}\right)^n e^{\frac{1}{12n}-\frac{1}{360n^3+103n}}
\ge n!\ge \sqrt{2\pi n}\left(\frac{n}{e}\right)^n e^{\frac{1}{12n}-\frac{1}{360n^3+102n}}$$
and for all $n\ge 3$
$$\sqrt{2\pi n}\left(\frac{n}{e}\right)^n e^{\frac{1}{12n+\frac{2}{5n}-\frac{1.1}{10n^3}}}
\ge n!\ge \sqrt{2\pi n}\left(\frac{n}{e}\right)^n e^{\frac{1}{12n+\frac{2}{5n}-\frac{0.9}{10n^3}}}.$$
\end{abstract}
\maketitle

\section{Introduction}
There are many asymptotic approximations for factorials and few precise bounds. Robbins gave in
1955 the following precise bound,~\cite{R55},
$$\sqrt{2\pi n}\left(\frac{n}{e}\right)^ne^{\frac{1}{12n}}\ge  n!\ge \sqrt{2\pi n}\left(\frac{n}{e}\right)^ne^{\frac{1}{12n+1}}.$$
Based on Robbin's analysis Maria gave the following tighter bound,~\cite{M65},
$$\sqrt{2\pi n}\left(\frac{n}{e}\right)^ne^{\frac{1}{12n}}\ge n!\ge \sqrt{2\pi n}\left(\frac{n}{e}\right)^ne^{\frac{1}{12n+\frac{3}{2(2n+1)}}}.$$
We give an elementary analysis that achieves many other tighter bounds. In particular, we improve
the above two bounds to: for all $n\ge 3$
$$\sqrt{2\pi n}\left(\frac{n}{e}\right)^n e^{\frac{1}{12n+\frac{2}{5n}-\frac{1.1}{10n^3}}}
\ge n!\ge \sqrt{2\pi n}\left(\frac{n}{e}\right)^n e^{\frac{1}{12n+\frac{2}{5n}-\frac{0.9}{10n^3}}}.$$
In \cite{I03}, Impens gave other precise bounds that we also improve with our elementary analysis. See the bounds in the table of Section~\ref{Other}.

In the literature, there are other precise bounds that are proved using non-elementary mathematics. See for example~\cite{SS90,N15}.

\section{Preliminary Results}
In this section we give some preliminary results.
\begin{lemma}\label{L1} For any sequence $a_1,a_2,\ldots$ of positive read numbers, if for $n\ge n_0$
\begin{eqnarray}\label{GE}
\left(n+\frac{1}{2}\right) \ln \left(1+\frac{1}{n}\right)-1 \ge a_n-a_{n+1}
\end{eqnarray} (respectively, $\le$) then for any $k\ge n\ge n_0$
we have
$$\frac{n!e^n}{n^{n+(1/2)}}\cdot e^{-a_{n}}\ge \frac{k!e^k}{k^{k+(1/2)}}\cdot e^{-a_{k}}$$
(respectively, $\le$).
\end{lemma}
\begin{proof} It is enough to prove that for every $n\ge n_0$
$$\frac{n!e^n}{n^{n+(1/2)}}\cdot e^{-a_{n}}\ge \frac{(n+1)!e^{n+1}}{(n+1)^{(n+1)+(1/2)}}\cdot e^{-a_{n+1}}.$$ This is equivalent to
$$\left(1+\frac{1}{n}\right)^{n+\frac{1}{2}}\ge e^{1+a_n-a_{n+1}}$$ which is equivalent to (\ref{GE}).\qed
\end{proof}

In the Appendix we give a sketch of the proof of Wallis' formula and to the fact that
$$\lim_{n\to\infty} \frac{n!e^n}{n^{n+\frac{1}{2}}}=\sqrt{2\pi}.$$
Now we prove
\begin{lemma}\label{main} For any sequence $a_1,a_2,\ldots$ of positive real numbers, where $\lim_{n\to\infty}a_n=0$, if for $n\ge n_0$
$$\left(n+\frac{1}{2}\right) \ln \left(1+\frac{1}{n}\right)-1 \ge a_n-a_{n+1}$$
(respectively, $\le$) then for $n\ge n_0$
$$n!\ge \sqrt{2\pi n}\left(\frac{n}{e}\right)^n e^{a_n}$$
(respectively, $\le$).
\end{lemma}
\begin{proof} By Lemma~\ref{L1}, for $n\ge n_0$,
$$\frac{n!e^n}{n^{n+1/2}}\cdot e^{-a_{n}}\ge \lim_{k\to\infty}\frac{k!e^k}{k^{k+1/2}}\cdot e^{-a_{k}}=\sqrt{2\pi}.$$ Therefore
$$n!\ge \sqrt{2\pi n}\left(\frac{n}{e}\right)^n e^{a_n}.\qed$$
\end{proof}
Notice that
\begin{eqnarray}\label{Taylor}\left(n+\frac{1}{2}\right) \ln \left(1+\frac{1}{n}\right)-1=
\sum_{k=2}^\infty \frac{(-1)^k(k-1)}{2k(k+1)}\frac{1}{n^k}.
\end{eqnarray}

Since ${(-1)^k(k-1)}/({2k(k+1)}{n^k})$ is a monotonically decreasing sequence we have
$$\sum_{k=2}^{2r-1} \frac{(-1)^k(k-1)}{2k(k+1)}\frac{1}{n^k}\le \sum_{k=2}^\infty \frac{(-1)^k(k-1)}{2k(k+1)}\frac{1}{n^k}\le \sum_{k=2}^{2r} \frac{(-1)^k(k-1)}{2k(k+1)}\frac{1}{n^k}.$$
This with Lemma~\ref{main} and (\ref{Taylor}) implies
\begin{lemma}\label{FIN} For any sequence $a_1,a_2,\ldots$ of positive real numbers where $\lim_{n\to\infty}a_n=0$ and $r>1$, if for $n\ge n_0$
\begin{eqnarray}\label{greater2}
\sum_{k=2}^{2r-1} \frac{(-1)^k(k-1)}{2k(k+1)}\frac{1}{n^k} \ge a_n-a_{n+1}
\end{eqnarray}
then for $n\ge n_0$
\begin{eqnarray}\label{greater}
n!\ge \sqrt{2\pi n}\left(\frac{n}{e}\right)^n e^{a_n}.
\end{eqnarray}
If for $n\ge n_0$
\begin{eqnarray}\label{less2}
\sum_{k=2}^{2r} \frac{(-1)^k(k-1)}{2k(k+1)}\frac{1}{n^k} \le a_n-a_{n+1}
\end{eqnarray}
then for $n\ge n_0$
\begin{eqnarray}\label{less}
n!\le \sqrt{2\pi n}\left(\frac{n}{e}\right)^n e^{a_n}.
\end{eqnarray}
\end{lemma}
Notice that
\begin{eqnarray*}
\sum_{k=2}^{\infty} \frac{(-1)^k(k-1)}{2k(k+1)}\frac{1}{n^k} &=&\frac{1}{12}\frac{1}{n^2}-\frac{1}{12}\frac{1}{n^3}+
\frac{3}{40}\frac{1}{n^4}-\frac{1}{15}\frac{1}{n^5}+
\frac{5}{84}\frac{1}{n^6}-\frac{3}{56}\frac{1}{n^7} \\
&&+\frac{7}{144}\frac{1}{n^8}-\frac{2}{45}\frac{1}{n^9}+\frac{9}{220}\frac{1}{n^{10}}-
\frac{5}{132}\frac{1}{n^{11}}+\frac{11}{312}\frac{1}{n^{12}}\cdots
\end{eqnarray*}

\section{Bounds for $n!$}
In this section we prove some bounds for $n!$ using Lemma~\ref{FIN}.

We first improve Robbins and Maria's bound~\cite{M65,R55}. Take $r=4$ and $a_n=1/(12n+0.4/n+0.09/n^3)$. Then
$$\frac{1}{12}\frac{1}{n^2}-\frac{1}{12}\frac{1}{n^3}+
\frac{3}{40}\frac{1}{n^4}-\frac{1}{15}\frac{1}{n^5}+\frac{5}{84}\frac{1}{n^6}-\frac{3}{56}\frac{1}{n^7}\ge\ \ \ \ \ \ \ \  $$ $$\ \ \ \ \ \ \ \  \left(\frac{1}{12n+\frac{2}{5n}+\frac{.9}{10n^3}}\right)-\left(\frac{1}{12(n+1)+\frac{2}{5(n+1)}+\frac{.9}{10(n+1)^3}}\right)$$
is equivalent to \begin{eqnarray*}
460000n^9+460000n^8-62970400n^7-181440000n^6-191576090n^5\\
\ \ \ \ \ \ \ \ -72519910n^4-5874457n^3-859176n^2+1422450n+498555\ge 0\end{eqnarray*} which is true for $n\ge 13$.
By Lemma~\ref{FIN} and by verifying that (\ref{greater12}) is also true for $n=3,4,\ldots,12$ we get that for every $n\ge 3$
\begin{eqnarray}\label{greater12}
n!\ge \sqrt{2\pi n}\left(\frac{n}{e}\right)^n e^{\frac{1}{12n+\frac{2}{5n}+\frac{.9}{10n^3}}}.
\end{eqnarray}
On the other hand for $r=4$ and $a_n=1/(12n+0.4/n+0.11/n^3)$
$$\frac{1}{12}\frac{1}{n^2}-\frac{1}{12}\frac{1}{n^3}+
\frac{3}{40}\frac{1}{n^4}-\frac{1}{15}\frac{1}{n^5}+\frac{5}{84}\frac{1}{n^6}-\frac{3}{56}\frac{1}{n^7}+\frac{7}{144}\frac{1}{n^8}\le\ \ \ \ \ \ \ \  $$ $$\ \ \ \ \ \ \ \  \left(\frac{1}{12n+\frac{2}{5n}+\frac{1.1}{10n^3}}\right)-\left(\frac{1}{12(n+1)+\frac{2}{5(n+1)}+\frac{1.1}{10(n+1)^3}}\right)$$
is equivalent to
\begin{eqnarray*}
-2280000n^{10}-2280000n^9-29928000n^8+322560000n^7+990219780n^6\\
+1047284220n^5+394378298n^4+31555984n^3+2970300n^2-9501470n\\
-3312155\le 0
\end{eqnarray*}
which is true for $n\ge 6$. By Lemma~\ref{FIN} and by verifying that (\ref{less34}) is also true for $n=1,2,\ldots,5$ we get that for every $n\ge 1$
\begin{eqnarray}\label{less34}
n!\le \sqrt{2\pi n}\left(\frac{n}{e}\right)^n e^{\frac{1}{12n+\frac{2}{5n}+\frac{1.1}{10n^3}}}.
\end{eqnarray}


Now for $r=4$ and $a_n=1/(12n)-1/(360n^3+103n)$,
 $$\frac{1}{12}\frac{1}{n^2}-\frac{1}{12}\frac{1}{n^3}+
\frac{3}{40}\frac{1}{n^4}-\frac{1}{15}\frac{1}{n^5}+
\frac{5}{84}\frac{1}{n^6}-\frac{3}{56}\frac{1}{n^7}
+\frac{7}{144}\frac{1}{n^8}\le$$
$$\ \ \ \ \ \ \ \  \left(\frac{1}{12n}-\frac{1}{360n^3+103n}\right)-\left(\frac{1}{12(n+1)}-\frac{1}{360(n+1)^3+103(n+1)}\right)$$ is equivalent to
\begin{eqnarray*}
-3600 n^7-1687578 n^5+30717978 n^4+58917996 n^3+\\
\ \ \ \ \ \ \ \ \ \ \ \ \ \ \ \ \ \ 49497870 n^2+16976975 n+11683805\le 0
\end{eqnarray*}
which is true for $n\ge 14$. By Lemma~\ref{FIN} and by verifying that (\ref{less}) is also true for $n=1,2,\ldots,14$ we get that for every $n\ge 1$
$$n!\le \sqrt{2\pi n}\left(\frac{n}{e}\right)^n e^{\frac{1}{12n}-\frac{1}{360n^3+103n}}.$$

On the other hand, take $r=5$ and $a_n=1/(12n)-1/(360n^3+102n)$. Then
$$\frac{1}{12}\frac{1}{n^2}-\frac{1}{12}\frac{1}{n^3}+
\frac{3}{40}\frac{1}{n^4}-\frac{1}{15}\frac{1}{n^5}+\frac{5}{84}\frac{1}{n^6}-\frac{3}{56}\frac{1}{n^7}
+\frac{7}{144}\frac{1}{n^8}-\frac{2}{45}\frac{1}{n^9}\ge\ \ \ \ \ \ \ \  $$ $$\ \ \ \ \ \ \ \  \left(\frac{1}{12n}-\frac{1}{360n^3+102n}\right)-\left(\frac{1}{12(n+1)}-\frac{1}{360(n+1)^3+102(n+1)}\right)$$ is equivalent to
\begin{eqnarray*}600n^8-46338n^6+46338n^5-782124n^4-1506090n^3-\\
\ \ \ \ \ \ \ \ \ \ \ \ \ \ \ \ \ \ \ 1253245n^2-429471n-293216\ge 0
\end{eqnarray*}
which is true for all $n\ge 10$. By Lemma~\ref{FIN} and by verifying that (\ref{greater}) is also true for $n=8,9$ we get that for every $n\ge 8$
$$n!\ge \sqrt{2\pi n}\left(\frac{n}{e}\right)^n e^{\frac{1}{12n}-\frac{1}{360n^3+102n}}.$$

\section{Other Bounds}\label{Other}
The following table gives other precise bounds for $n!$
\begin{center}
\begin{tabular}{|c|}
\hline
$r=5$, $a_n=\frac{1}{12n}-\frac{1}{360n^3}+\frac{1}{1260n^5+944n^3}$\\
\hline
{\footnotesize $-24255n^9-24255n^8+19534030n^7+208372500n^6+846744589n^5+1608743411n^4+$}\\
{\footnotesize $1838090736n^3+1481505592n^2+901562480n+294921648\le 0$ for $n\ge 33$}\\
\hline\\
$n!\le \sqrt{2\pi n}\left(\frac{n}{e}\right)^n e^{\frac{1}{12n}-\frac{1}{360n^3}+\frac{1}{1260n^5+944n^3}}$ for $n\ge 26$\\[-5pt] \\
\hline \hline
$r=6$, $a_n=\frac{1}{12n}-\frac{1}{360n^3}+\frac{1}{1260n^5+945n^3}$\\
\hline
{\footnotesize $3156 n^8-31463 n^6-126937 n^5-241045 n^4-275373 n^3-221928 n^2-$}\\
{\footnotesize $135072 n-44100\ge 0$ for $n\ge 5$}\\
\hline\\
$n!\ge \sqrt{2\pi n}\left(\frac{n}{e}\right)^n e^{\frac{1}{12n}-\frac{1}{360n^3}+\frac{1}{1260n^5+945n^3}}$ for $n\ge 1$\\[-5pt]  \\
\hline \hline
$r=6$, $a_n=\frac{1}{12n}-\frac{1}{360n^3}+\frac{1}{1260n^5}-\frac{1}{1680n^7+2376n^5}$\\
\hline
{\footnotesize $-2730n^{11}-5460n^{10}-28258433n^9+88168297n^8+701534344n^7+2112056100n^6$}\\
{\footnotesize $+4069612325n^5+5596290735n^4+5773252968n^3+4320089004n^2$}\\
{\footnotesize $+2021099850n+425134710\le 0$ for $n\ge 8$}\\
\hline\\
$n!\le \sqrt{2\pi n}\left(\frac{n}{e}\right)^n e^{\frac{1}{12n}-\frac{1}{360n^3}+\frac{1}{1260n^5}\frac{1}{1680n^7+2376n^5}}$ for $n\ge 1$\\[-5pt]  \\
\hline \hline
$r=7$, $a_n=\frac{1}{12n}-\frac{1}{360n^3}+\frac{1}{1260n^5}-\frac{1}{1680n^7+2375n^5}$\\
\hline
{\footnotesize $196560n^{12}+393120n^{11}-650113107n^{10}-650309667n^9-2613399138n^8-15256200960n^7-$}\\
{\footnotesize $45641758349n^6-87900450451n^5-120821404840n^4-124589009460n^3-$}\\
{\footnotesize $93179955210n^2-43560354750n-9152946000\ge 0$ for $n\ge 58$}\\
\hline\\
$n!\ge \sqrt{2\pi n}\left(\frac{n}{e}\right)^n e^{\frac{1}{12n}-\frac{1}{360n^3}+\frac{1}{1260n^5}\frac{1}{1680n^7+2375n^5}}$ for $n\ge 53$\\[-5pt]  \\
\hline \hline
\end{tabular}
\end{center}

In  each one of the above tables, the first row gives $r$ and $a_n$ that is used in Lemma~\ref{FIN}. The second row is the inequality for which (\ref{greater2}) or (\ref{less2}) is equivalent and for which $n$ this inequality is valid. The third row is the result.
\ignore{
\section{TO CHECK IF WE SHOULD ADD IT}
\section{Exact}
From Lemma if
$$\left(n+\frac{1}{2}\right) \ln \left(1+\frac{1}{n}\right)-1 =A_n=
\sum_{k=2}^\infty \frac{(-1)^k(k-1)}{2k(k+1)}\frac{1}{n^k}= a_n-a_{n+1}$$
then
$$n!= \sqrt{2\pi n}\left(\frac{n}{e}\right)^n e^{a_n}.$$
Notice that $a_n=A_n+A_{n+1}+\cdots$ satisfy $a_n-a_{n+1}=A_n$. Therefore
\begin{eqnarray*}
a_n&=&\sum_{i=0}^\infty\sum_{k=2}^\infty \frac{(-1)^k(k-1)}{2k(k+1)}\frac{1}{(n+i)^k}.\\
&=& \sum_{k=2}^\infty \frac{(-1)^k(k-1)}{2k(k+1)}\sum_{i=0}^\infty \frac{1}{(n+i)^k}.\\
\end{eqnarray*}

This can be bounded by some integrals}

\newpage
\section{APPENDIX}
Here we give a sketch of the proof of Wallis' formula and to fact that
$$\lim_{n\to\infty} \frac{n!e^n}{n^{n+\frac{1}{2}}}=\sqrt{2\pi}.$$
Consider $n\ge 0$ and
$$I_n=\int_0^{\pi/2}\sin^nx dx.$$
Obviously, $\sin^nx\le \sin^{n-1}x$ and therefore $I_n<I_{n-1}$ for all $n\ge 1$.
By the integration by parts it is easy to see that $I_0=\pi/2$, $I_1=1$
$$I_n=\frac{n-1}{n} I_{n-2}.$$ This implies
$$I_{2n}=\frac{(2n)!}{2^{2n}(n!)^2}\cdot\frac{\pi}{2},\ \ \ \ \ \ I_{2n+1}=\frac{2^{2n}(n!)^2}{(2n+1)!}.$$
From the inequality $I_{2n-1}>I_{2n}>I_{2n+1}$ we get
$$\sqrt{\pi n}< \frac{2^{2n}(n!)^2}{(2n)!}< \sqrt{\pi (n+1/2)}.$$
By Lemma~\ref{L1} and (\ref{Taylor}), since for $a_n=0$
$$
\sum_{k=2}^\infty \frac{(-1)^k(k-1)}{2k(k+1)}\frac{1}{n^k}\ge 0$$
we have for every $k>n$,
$$\frac{n!e^n}{n^{n+1/2}}\ge \frac{k!e^k}{k^{k+1/2}}.$$ In particular for $k=2n$ this inequality is equivalent to
$$\frac{n!e^n}{n^{n+1/2}}\le \sqrt{1+\frac{1}{2n}}\sqrt{2\pi}.$$
Similarly, since for $a_n=1/n$,
$$
\sum_{k=2}^\infty \frac{(-1)^k(k-1)}{2k(k+1)}\frac{1}{n^k}\le \frac{1}{12n^2}\le \frac{1}{n}-\frac{1}{n+1}$$
$$\frac{n!e^n}{n^{n+1/2}}\cdot e^{-1/n}\ge \frac{k!e^k}{k^{k+1/2}}\cdot e^{-1/k}.$$ In particular for $k=2n$ this inequality is equivalent to
$$\frac{n!e^n}{n^{n+1/2}}\ge e^{\frac{1}{2n}-\frac{1}{n}}\sqrt{2\pi}.$$
Then by the sandwich theorem we get
$$\lim_{n\to\infty} \frac{n!e^n}{n^{n+\frac{1}{2}}}=\sqrt{2\pi}.$$


\begin{thebibliography}{}
\bibitem{I03}
C. Impense.
Stirling's Series Made Easy.
{\it The American Mathematical Monthly}, 110 (8): pp. 730--735.(2003).

\bibitem{M65}
A. J. Maria.
A Remark on Stirling's Formula.
{\it The American Mathematical Monthly}, 72 (10): pp. 1096--1098.(1965).


\bibitem{N15}
G. Nemes. Error bounds and exponential improvements for the asymptotic expansions of the gamma function and its reciprocal, Proc. Roy. Soc. Edinburgh Sect. A 145 (2015), 571–596.


\bibitem{R55}
H. Robbins.
A Remark on Stirling's Formula.
{\it The American Mathematical Monthly}, 62 (1): pp. 26--29.(1955).

\bibitem{SS90}
F. W. Sch\"afke, A. Sattler, Restgliedabschätzungen für die Stirlingsche Reihe. Note. Mat. 10 (1990), 453–470.



\end{thebibliography}
\end{document}